\documentclass[12pt]{article}
\usepackage{fullpage}
\usepackage{amssymb,amsmath,amsfonts,amsthm,latexsym,enumerate,url,cases}
\numberwithin{equation}{section}
\usepackage{hyperref}
\hypersetup{colorlinks=true,citecolor=blue,linkcolor=blue,urlcolor=blue}

\newtheorem{theorem}{Theorem}[section] %
\newtheorem{lemma}[theorem]{Lemma} %
\newtheorem{remark}[theorem]{Remark} %

\begin{document}
\title{Diophantine equations \\ involving Euler's totient function}

\author{  Yong-Gao Chen\footnote{Corresponding author, E-mail: ygchen@njnu.edu.cn(Y.-G. Chen)}  and Hao Tian \\
\small  School of Mathematical Sciences and Institute of Mathematics, \\
\small  Nanjing Normal University,  Nanjing  210023,  P. R. China
}
\date{}
\maketitle \baselineskip 18pt \maketitle \baselineskip 18pt

\begin{abstract}   In this paper, we consider the equations involving Euler's totient function $\phi $
 and Lucas type sequences. In particular,
we prove that the equation $\phi (x^m-y^m)=x^n-y^n$ has no
solutions in positive integers $x, y, m, n$ except for the trivial
solutions $(x, y, m , n)=(a+1, a, 1, 1)$, where $a$ is a positive
integer, and the equation $\phi ((x^m-y^m)/(x-y))=(x^n-y^n)/(x-y)$
has no solutions in positive integers $x, y, m, n$ except for the
trivial solutions $(x, y, m , n)=(a, b, 1, 1)$, where $a, b$ are
integers with $a>b\ge 1$.

\medskip

{\noindent\bf 2010 MR Subject Classification:} 11A25, 11D61, 11D72

\end{abstract}

\section{Introduction}

There are many famous problems on Euler's totient function $\phi
$. For example, the equation $\phi (n)=\phi (n+k)$ has brought
many interests (see  Ballew, Case and Higgins
\cite{BallewCaseHiggins1975}, Holt \cite{Holt2003}, Lal and
Gillard \cite{LalandGillard1972}, Schinzel \cite{Schinzel1958}).
In 1932, Lehmer \cite{Lehmer1932} asked whether there are
composite numbers $n$ for which  $n-1$ is divisible by $\phi (n)$.
In 1922, Carmichael \cite{Carmichael1922} conjectured that, for
every positive integer $n$, there exists a positive integer
$m\not= n$ such that $\phi (m)=\phi (n)$. For related progress,
one may see Banks etc
\cite{BanksFriedlanderLucaPappalardiShparlinski2006}, Bateman
\cite{Bateman1972}, Contini, Croot and Shparlinski
\cite{ContiniCrootShparlinski2006}, Erd\H os \cite{Erdos1945} and
\cite{Erdos1958}, Erd\H os and Hall
\cite{ErdosHall1973,ErdosHall1976,ErdosHall1977},  Ford
\cite{Ford1999}, Guderson \cite{Guderson1943}, Guy
\cite[B36-B42]{Guy2004}, Pomerance \cite{Pomerance1980}, and
Rotkiewicz \cite{Rotkiewicz1961}.

 In this paper, we consider the equations involving Euler's totient function $\phi $ and Lucas type sequences.
 In particular, we consider the following equations
\begin{equation}\label{eqn1} \phi\left( x^m-y^m \right) =x^n-y^n\end{equation}
and
\begin{equation}\label{eqn1x} \phi\left( \frac{x^m-y^m}{x-y} \right) =\frac{x^n-y^n}{x-y}\end{equation}
in positive integers $x, y, m, n$.

Luca \cite{Luca2008} proved that, if $b\ge 2$ is a fixed integer,
then the equation
$$\phi \left( x \frac{b^m-1}{b-1} \right) = y
\frac{b^n-1}{b-1},\quad x, y\in \{ 1, 2, \dots , b-1\} $$ has only
finitely many positive integer solutions $(x, y, m, n)$. In 2015,
Faye and Luca \cite{Luca2015} proved that, if $(m, n, x)$ is a
 solution of
\begin{equation}\label{eqn1xxx}\phi\left( x^m-1 \right) =x^n-1\quad  \text{ or } \quad  \phi\left( \frac{x^m-1}{x-1} \right)
=\frac{x^n-1}{x-1}\end{equation} in positive integers $x, m, n$
with $m>n$, then
$$x<e^{e^{8000}}.$$ Currently,  even for a given odd number
$x>2$, there is no method to know rapidly whether the equations of
\eqref{eqn1xxx}
 have solutions in positive integers $ m, n$
with $m>n$. In 2015, Faye,  Luca and Tall \cite{Luca2015a} proved
that the equation
\begin{equation*}\phi(5^m-1)=5^n-1\end{equation*}
has no solutions  in positive integers $m, n$. This solves a
problem in \cite{Luca1997}.

In this paper, the following results are proved.

\begin{theorem}\label{thm0} The equation \eqref{eqn1} has no solutions in positive integers $x, y,
m, n$ except for the trivial solutions $(x, y, m , n)=(a+1, a, 1,
1)$, where $a$ is a positive integer.\end{theorem}

\begin{theorem}\label{thm1} The  equation  \eqref{eqn1x} has no
solutions in positive integers $x, y, m, n$ except for the trivial
solutions $(x, y, m , n)=(a, b, 1, 1)$, where $a, b$ are integers
with $a>b\ge 1$.\end{theorem}

\section{Generalization and Crucial Reduction}

In this paper we will concern the following equation:
\begin{equation}\label{eqn1a}\phi
\left( z\frac{x^m-y^m}{x-y}\right)
=z\frac{x^n-y^n}{x-y}\end{equation} in positive integers $x, y, z,
m, n$ with $x>y$. Now Theorems \ref{thm0} and \ref{thm1} are two
special cases for $z=x-y$ and $z=1$, respectively.

It is clear that  $$(x, y, z, m, n)=(a, b, 1, 1, 1), \quad a, b\in
\mathbb{Z}^+, a>b\ge 1$$ are  solutions of \eqref{eqn1a}. Such
solutions are called the trivial solutions of \eqref{eqn1a}. Since
$\phi (k)=k$ if and only if $k=1$, it follows that $m>n$ if $(x,
y, z, m, n)$ is a nontrivial solution of \eqref{eqn1a}.

For the equation \eqref{eqn1a}, we have the following result.

\begin{theorem}\label{thm2} The  equation \eqref{eqn1a} has no
nontrivial solutions in positive integers $x, y, z, m, n$ with
$1\le z\le x-y$.\end{theorem}

Theorems \ref{thm0} and \ref{thm1} follow from Theorem \ref{thm2}
by taking $z=x-y$ and $z=1$, respectively. It is interesting that
it is difficult directly to give proofs of Theorems \ref{thm0} and
\ref{thm1}. But general Theorem \ref{thm2} is ``easily" proved.
Our another key observation is to find that we may assume that
$\gcd (m, n)=1$. Theorem \ref{thm2} is equivalent to the following
theorem.

\begin{theorem}\label{thm3} The  equation \eqref{eqn1a} has no
nontrivial solutions in positive integers $x, y, z, m, n$ with
$1\le z\le x-y$ and $\gcd (m, n)=1$.\end{theorem}

Now we prove that Theorem \ref{thm2} and Theorem \ref{thm3} are
equivalent each other. It is clear that Theorem \ref{thm2} implies
Theorem \ref{thm3}.

Suppose that Theorem \ref{thm3} is true and $(x, y, z, m, n)$ is a
nontrivial solution of the equation \eqref{eqn1a} in positive
integers $x, y, z, m, n$ with $1\le z\le x-y$. Then $m>n\ge 1$.
Let $$\gcd (m,n)=d',\quad  m=d'm',\quad n=d'n',\quad
x^{d'}=x',\quad y^{d'}=y'$$ and
$$z\frac{x^{d'}-y^{d'}}{x-y} =z'.$$
Then $\gcd (m', n')=1$, $m'>n'\ge 1$,
\begin{equation}\label{ab}\phi
\left( z'\frac{(x')^{m'}-(y')^{m'}}{x'-y'}\right)
=z'\frac{(x')^{n'}-(y')^{n'}}{x'-y'}\end{equation} and
$$1\le z'=z\frac{x^{d'}-y^{d'}}{x-y}\le x^{d'}-y^{d'}=x'-y'.$$
Thus $(x', y', z', m', n')$ is a nontrivial solution of the
equation \eqref{eqn1a} in positive integers $x', y', z', m', n'$
with $1\le z'\le x'-y'$ and $\gcd (m', n')=1$. This contradicts
Theorem \ref{thm3}. Now we have proved that Theorem \ref{thm2} and
Theorem \ref{thm3} are equivalent each other.

In the following, our task is to prove Theorem \ref{thm3}. One may
see that the condition $\gcd (m, n)=1$ plays a key role in our
proof.

{\it From now on, we always assume that $(x, y, z, m, n)$ is a
nontrivial solution of the  equation \eqref{eqn1a} in positive
integers $x, y, z, m, n$ with $1\le z\le x-y$ and $\gcd (m, n)=1$.
Since $(x, y, z, m, n)$ is a nontrivial solution of the equation
\eqref{eqn1a}, it follows that $m>n$. Let
$$\gcd (x, y)=d_1,\quad x=x_1d_1,\quad y=y_1d_1.$$
In this paper, $p$, $q$, $r$ and $\gamma$ always denote odd
primes. Let $p(m)$ be the least prime divisor of $m$. For each
prime $p\ge 3$ with $p\nmid x_1y_1$, let $\ell_p$ be the least
positive integer $\ell $
  such that $p\mid x_1^\ell -y_1^\ell $.
Then $p\mid x_1^m-y_1^m$ if and only if $\ell_p \mid m$. For the
convenience of the reader, we will repeat some statements in the
proof.}

 In Section
\ref{seca}, we solve the  equation \eqref{eqn1a} in positive
integers $x, y, z, m, n$ with $x_1$ and $y_1$ having different
parities and no constrains on the size of $z$. It follows that
Theorem \ref{thm3} is true for $x_1$ and $y_1$ having different
parities. Since $\gcd (x_1, y_1)=1$, we may assume that $x_1$ and
$y_1$ are both odd. In Section \ref{sec1}, we  give the
preliminary lemmas. We prove Theorem \ref{thm3} in two sections.
In Section \ref{sec2}, we prove Theorem \ref{thm3} for $x>80$. For
this, we divide into two subsections: $p(m)\le x$ and $p(m)>x$. In
Section \ref{sec3}, we prove Theorem \ref{thm3} for $x\le 80$.

\section{The equation without constrains on $z$ }\label{seca}

For any prime $p$ and any positive integer $a$, let $\nu_p (a)$
denote the integer $k$ with
$$p^k \mid a,\quad p^{k+1} \nmid a.$$

\begin{theorem}\label{thm2a} The only nontrivial solutions of the equation
\eqref{eqn1a}  in positive integers $x, y, z, m, n$ with $\nu_2
(x)\not= \nu_2(y)$ are $$(x, y, z, m , n)=(2, 1, 2^\beta p^u, q,
q-1),$$ where $q, p=2^q-1$  are both primes and $\beta$, $u$ are
two integers with $\beta \ge 1$ and $u\ge 0$.\end{theorem}

\begin{proof} Suppose that $(x, y, z, m ,
n)$ is a  nontrivial solution of the equation \eqref{eqn1a}  in
positive integers $x, y, z, m, n$. Then $m>n\ge 1$. As in the
previous section, first we reduce the problem to the case $\gcd
(m, n)=1$.

Let
$$\gcd (m,n)=d',\quad m=d'm',\quad n=d'n',\quad x^{d'}=x',\quad
y^{d'}=y'$$ and
$$z\frac{x^{d'}-y^{d'}}{x-y} =z'.$$
Then
\begin{equation}\label{ab}\phi
\left( z'\frac{(x')^{m'}-(y')^{m'}}{x'-y'}\right)
=z'\frac{(x')^{n'}-(y')^{n'}}{x'-y'}\end{equation} and $\nu_2
(x')=d'\nu_2(x)\not= d'\nu_2 (y)=\nu_2 (y')$. Suppose that the
only nontrivial solutions of the equation \eqref{ab}  in positive
integers $x', y', z', m', n'$ with $\gcd (m', n')=1$ are $$(x',
y', z', m' , n')=(2, 1, 2^\beta p^u, q, q-1),$$ where $q, p=2^q-1$
are both primes and $\beta$, $u$ are nonnegative integers with
$\beta \ge 1$. Then $x^{d'}=x'=2$. So $d'=1$ and $x=2$. Thus the
only nontrivial solutions of the equation \eqref{eqn1a}  in
positive integers $x, y, z, m, n$ are $$(x, y, z, m , n)=(2, 1,
2^\beta p^u, q, q-1),$$ where $q, p=2^q-1$ are both primes and
$\beta$, $u$ are nonnegative integers with $\beta \ge 1$. Now we
have reduced the problem to the case $\gcd (m, n)=1$. Without loss
of generality, we may assume that $\gcd (m, n)=1$.

Let
$$\gcd (x, y)=d_1,\quad x=x_1d_1,\quad y=y_1d_1.$$
 Now
\eqref{eqn1a} becomes
\begin{equation}\label{eqw1}\phi
\left( zd_1^{m-1} \frac{x_1^m-y_1^m}{x_1-y_1}\right) = zd_1^{n-1}
\frac{x_1^n-y_1^n}{x_1-y_1}.\end{equation} Since $\nu_2 (x)\not=
\nu_2(y)$, it follows that $x_1$ and $y_1$ have different
parities. Let $$ z=2^\beta t_1,\quad d_1=2^\alpha w_1,\quad 2\nmid
t_1w_1.$$ Now \eqref{eqw1} becomes
\begin{equation*}\phi
\left( 2^{\beta +\alpha (m-1)}t_1 w_1^{m-1}
\frac{x_1^m-y_1^m}{x_1-y_1}\right) =2^{\beta +\alpha (n-1)}t_1
w_1^{n-1} \frac{x_1^n-y_1^n}{x_1-y_1}.\end{equation*} Since $x_1$
and $y_1$ have different parities, it follows that
$$\frac{x_1^m-y_1^m}{x_1-y_1}$$
and
$$\frac{x_1^n-y_1^n}{x_1-y_1}$$
are both odd. Let
\begin{equation}\label{ab1xx}A=t_1 w_1^{m-1}
\frac{x_1^m-y_1^m}{x_1-y_1},\quad B=t_1 w_1^{n-1}
\frac{x_1^n-y_1^n}{x_1-y_1}.\end{equation}  Then $A$ and $B$ are
both odd and \begin{equation}\label{ab1}\phi \left( 2^{\beta
+\alpha (m-1)}A\right) =2^{\beta +\alpha (n-1)}B.\end{equation} By
$m>n\ge 1$ and \eqref{ab1xx}, we have $$A=t_1 w_1^{m-1}
\frac{x_1^m-y_1^m}{x_1-y_1}\ge \frac{x_1^m-y_1^m}{x_1-y_1}>1 .$$
Noting that $A$ is odd, $\phi (A)$ is even. We divide into two
cases:

{\bf Case 1:} $\alpha =\beta =0$. Then \eqref{ab1} becomes $\phi
(A)=B$. Since $A$ and $B$ are both odd, it follows that $A=B=1$, a
contradiction with $A>1$.

{\bf Case 2:} $\alpha +\beta >0$. Then \eqref{ab1} becomes
\begin{equation}\label{ab1x}2^{\beta +\alpha (m-1)-1} \phi (A)=2^{\beta +\alpha (n-1)}B.\end{equation}
Since $2\mid \phi (A)$ and $B$ is odd, it follows that
$$\beta +\alpha (m-1)\le\beta +\alpha (n-1).$$
Noting that $m>n$, we have $\alpha =0$. Thus $\beta =\alpha +\beta
\ge 1$ and \eqref{ab1x} becomes $\phi (A)=2B$. Hence there exist
an odd prime $p$ and a positive integer $t$ such that
$$A=p^t, \quad 2B=p^{t-1} (p-1).$$
By $$p^t=A=t_1 w_1^{m-1} \frac{x_1^m-y_1^m}{x_1-y_1}, \quad
p^{t-1} (p-1)=2B=2t_1 w_1^{n-1} \frac{x_1^n-y_1^n}{x_1-y_1}, $$
there exist nonnegative integers $u,v, k$ such that
$$t_1=p^u, \quad w_1=p^v, \quad  \frac{x_1^m-y_1^m}{x_1-y_1}=p^k,\quad t=u+(m-1)v+k,$$
$$ p^{t-1-u-(n-1)v} (p-1)=2 \frac{x_1^n-y_1^n}{x_1-y_1}.$$ By $m>n\ge
1$, we have $k\ge 1$ and $t-1-u-(n-1)v\ge 0$. So $p\mid
x_1^m-y_1^m$.    If $$t-1-u-(n-1)v\ge 1,$$ then
$$p\mid x_1^n-y_1^n.$$
Thus
$$p\mid x_1^{(m,n)}-y_1^{(m,n)}.$$
That is, $p\mid x_1-y_1$. By $\gcd (x_1, y_1)=1$, $p\nmid x_1y_1$.
Since
\begin{eqnarray*}p^k&=&\frac{x_1^m-y_1^m}{x_1-y_1}\\
&=&x_1^{m-1}+x_1^{m-2}y_1+\cdots +y_1^{m-1}\\
&\equiv & my_1^{m-1} \pmod p\end{eqnarray*} and
\begin{eqnarray*}&&p^{t-1-u-(n-1)v}(p-1)\\
&=&2\frac{x_1^n-y_1^n}{x_1-y_1}\\
&=&2(x_1^{n-1}+x_1^{n-2}y_1+\cdots +y_1^{n-1})\\
&\equiv & 2 ny_1^{n-1} \pmod p,\end{eqnarray*} it follows that
$p\mid m$ and $p\mid n$, a contradiction with $\gcd (m, n)=1$. So
$$t-1-u-(n-1)v=0.$$ Noting that $t=u+(m-1)v+k$, we have
$$u+(m-1)v+k-1-u-(n-1)v=0.$$
That is, $(m-n)v+k-1=0$. Since $m>n$ and $k\ge 1$, it follows that
$v=0$ and $k=1$. Thus $d_1=2^\alpha w_1=1$,
$$p=\frac{x_1^m-y_1^m}{x_1-y_1}=\frac{x^m-y^m}{x-y},$$
and $$ p-1=2\frac{x_1^n-y_1^n}{x_1-y_1}=2\frac{x^n-y^n}{x-y}.$$ It
follows that
\begin{equation*}\frac{x^m-y^m}{x-y} -1=2 \frac{x^n-y^n}{x-y}.\end{equation*}
By $m>n$ and $x>y\ge 1$, we have
\begin{eqnarray*}2\frac{x^n-y^n}{x-y}&=&\frac{x^m-y^m}{x-y} -1\\
&=&y^m
\frac{(x/y)^m-1}{x-y}-1\\
&\ge & y^{n+1}
\frac{(x/y)^{n+1}-1}{x-y}-1\\
&= & \frac{x^{n+1}-y^{n+1}}{x-y}-1\\
&=& x^n+x^{n-1}y+\cdots +xy^{n-1}+y^n-1\\
&\ge & x^n+x^{n-1}y+\cdots +xy^{n-1}\\
&=& x\frac{x^n-y^n}{x-y}\\
&\ge & 2 \frac{x^n-y^n}{x-y}.
\end{eqnarray*}
It follows  that $x=2$, $y=1$ and $m=n+1$. Since
$$p=\frac{x^m-y^m}{x-y}=2^m-1$$
is a prime, it follows that $m$ is a prime. Write $m=q$. Then
$$(x, y, z, m , n)=(2, 1, 2^\beta p^u, q, q-1),$$ where $q,
p=2^q-1$  are both primes and $\beta$, $u$ are nonnegative
integers with $\beta \ge 1$. It is easy to verify that these are
solutions of the equation \eqref{eqn1a}. This completes the proof
of Theorem \ref{thm2a}.
\end{proof}

\section{Preliminary Lemmas}\label{sec1}

In this section, we give some preliminary lemmas. We always assume
that $(x, y, z, m, n)$ is a nontrivial solution of the  equation
\eqref{eqn1a} in positive integers $x, y, z, m, n$ with $1\le z\le
x-y$ and $\gcd (m, n)=1$. Then $m>n\ge 1$. Recall that
$$\gcd (x, y)=d_1,\quad x=x_1d_1,\quad y=y_1d_1.$$
If $x_1$ and $ y_1$ have different parities, then $\nu_2
(x)\not=\nu_2 (y)$. By Theorem \ref{thm2a}, there exist two primes
$q, p=2^q-1$ and two integers $\beta \ge 1$ and $u\ge 0$ such that
$(x, y, z, m , n)=(2, 1, 2^\beta p^u, q, q-1)$. This contradicts
$1\le z\le x-y$. Hence $x_1$ and $ y_1$ have the same parity.
Since $\gcd (x_1, y_1)=1$, it follows that $x_1$ and $ y_1$ are
both odd. Noting that $x_1>y_1\ge 1$, we have $x_1\ge 3$.

\begin{lemma}\label{lem1} Let $(x, y, z, m, n)$ be a
nontrivial solution of the  equation \eqref{eqn1a} in positive
integers $x, y, z, m, n$ with $1\le z\le x-y$ and $\gcd (m, n)=1$.
Then $2\nmid m$.
\end{lemma}

\begin{proof} We prove the lemma by a contradiction. Suppose
that $2\mid m$. Let $$m=2^\alpha m_1,\quad d_1=2^\beta w_1, \quad
z=2^\delta t_1,$$ where $\alpha , \beta , \delta , m_1, w_1 ,
t_1$ are nonnegative integers with $2\nmid m_1w_1t_1$ and $\alpha
\ge 1$. Let
$$A'=t_1 w_1^{m-1}\frac{x_1^m-y_1^m}{x_1-y_1}, \quad B'=t_1
w_1^{n-1}\frac{x_1^n-y_1^n}{x_1-y_1}.$$ Since $2\mid m$, we have
$2\nmid n$. It follows that $B'$ is odd. Since  $x_1$ and $y_1$
are odd,  $$A'=t_1 w_1^{m-1}\frac{x_1^m-y_1^m}{x_1-y_1}=t_1
w_1^{m-1}\frac{x_1^m-y_1^m}{x_1^2-y_1^2} (x_1+y_1)$$ is even and
more than $2$. Let $A'=2^\mu A_1$ with $2\nmid A_1$ and $\mu\ge
1$. By $A'>2$, we have either $\mu \ge 2$ or $A_1\ge 3$. Hence
$$\phi \left(
z\frac{x^m-y^m}{x-y}\right)=\phi ( 2^{\delta + \beta (m-1)+\mu}
A_1)=2^{\delta + \beta (m-1)+\mu -1} \phi (A_1)$$ is divisible by
$$2^{\delta + \beta (m-1)+1}.$$
By \eqref{eqn1a} and
$$z\frac{x^n-y^n}{x-y}= 2^{\delta + \beta (n-1)} B',$$
we have $\delta + \beta (m-1)+1\le \delta + \beta (n-1)$, a
contradiction with $m>n$.

Therefore, $2\nmid m$. This completes the proof of Lemma
\ref{lem1}.
\end{proof}

\begin{lemma}\label{lem3a} Let $q$ be a prime
divisor of $m$.  Then
$$q^{\frac 12 \alpha_q d(m)-1}\mid
x_1^{q-1}-y_1^{q-1},$$
 where $d(m)$ is the number of positive
divisors of $m$ and  $\alpha_q$ is the integer with $q^{\alpha_q}
\mid m$ and $q^{\alpha_q +1} \nmid m$.

 Furthermore,
if $q\nmid z$, then
$$q^{\frac 12 \alpha_q d(m)}\mid
x_1^{q-1}-y_1^{q-1}. $$
\end{lemma}

\begin{proof} Let $m=q^{\alpha_q} m_q$ and let $l_1, l_2, \dots , l_t$ be all positive divisors
of $m_q$. Then $q^i l_j$ $(1\le i\le \alpha_q , 1\le j\le t)$ are
$\alpha_q t$ distinct positive divisors of $m$. By Lemma
\ref{lem1}, $q^il_j\not= 2, 6$. By Carmichael's primitive divisor
theorem (see \cite{Carmichael1913}), each of $x_1^{q^i
l_j}-y_1^{q^i l_j}$
 has a primitive prime divisor $p_{i,j}\equiv 1\pmod{q^i
l_j}$. It is clear that
$$p_{i,j}\mid \frac{x_1^{q^i
l_j}-y_1^{q^i l_j}}{x_1-y_1},\quad \frac{x_1^{q^i l_j}-y_1^{q^i
l_j}}{x_1-y_1}\mid \frac{x_1^m-y_1^m}{x_1-y_1}.$$ It follows that
$$\prod_{\substack{1\le i\le \alpha_q\\ 1\le j\le t}} p_{i,j} \mid
\frac{x_1^m-y_1^m}{x_1-y_1}.$$ Let
$$z=q^{\delta_q} t_q,\quad d_1=q^{\beta_q} d_q,\quad  q\nmid t_q d_q.$$
Noting that $$z\frac{x^m-y^m}{x-y} =q^{\delta_q+\beta_q (m-1)} t_q
d_q^{m-1}\frac{x_1^m-y_1^m}{x_1-y_1},$$ we have
$$q^{\delta_q+\beta_q (m-1)} \prod_{\substack{1\le i\le \alpha_q\\ 1\le j\le t}} p_{i,j} \mid
z\frac{x^m-y^m}{x-y}.$$ So
$$\phi \left( q^{\delta_q+\beta_q (m-1)} \right) \prod_{\substack{1\le i\le \alpha_q\\ 1\le j\le t}} (p_{i,j}-1) \mid
\phi \left( z\frac{x^m-y^m}{x-y}\right).$$ It follows from
\eqref{eqn1a} that $$\phi \left( q^{\delta_q+\beta_q (m-1)}
\right)\prod_{\substack{1\le i\le \alpha_q\\ 1\le j\le t}}
(p_{i,j}-1) \mid z\frac{x^n-y^n}{x-y}.$$ So $$\phi \left(
q^{\delta_q+\beta_q (m-1)} \right) \prod_{\substack{1\le i\le
\alpha_q\\ 1\le j\le t}} q^i \mid z\frac{x^n-y^n}{x-y}.$$ That is,
$$\phi \left( q^{\delta_q+\beta_q (m-1)} \right) q^{\frac 12
\alpha_q (\alpha_q +1) t} \mid q^{\delta_q+\beta_q (n-1)} t_q
d_q^{n-1}\frac{x_1^n-y_1^n}{x_1-y_1}.$$ Noting that
$d(m)=(\alpha_q +1) t$, we have
\begin{equation}\label{a1}\phi \left( q^{\delta_q+\beta_q (m-1)}
\right) q^{\frac 12 \alpha_q d(m)} \mid q^{\delta_q+\beta_q (n-1)}
t_q d_q^{n-1}(x_1^n-y_1^n).\end{equation}

We divide into three cases:

{\bf Case 1:} $q\mid d_1$. Then $\beta_q \ge 1$.  By \eqref{a1},
$$q^{\delta_q+\beta_q (m-1)-1} q^{\frac 12 \alpha_q d(m)} \mid  q^{\delta_q+\beta_q (n-1)}(x_1^n-y_1^n).$$
That is,
$$q^{\beta_q (m-n)-1+\frac 12 \alpha_q d(m)} \mid  x_1^n-y_1^n.$$
Since $m>n$ and $\beta_q\ge 1$, it follows that
\begin{equation}\label{eqc2}q^{\frac 12 \alpha_q d(m)} \mid  x_1^n-y_1^n.\end{equation}
Since $\gcd (x_1, y_1)=1$, it follows from \eqref{eqc2} that
$q\nmid x_1y_1$. By Euler's theorem,
\begin{equation}\label{eqc3}q^{\frac 12 \alpha_q d(m)}  \mid
x_1^{q^{\frac 12 \alpha_q d(m) -1}(q-1)}-y_1^{q^{\frac 12 \alpha_q
d(m)-1}(q-1)}.\end{equation} In view of \eqref{eqc2} and
\eqref{eqc3},
$$ q^{\frac 12 \alpha_q d(m)}  \mid
x_1^{(q^{\frac 12 \alpha_q d(m) -1}(q-1), n)}-y_1^{(q^{\frac 12
\alpha_q d(m) -1}(q-1), n)}.$$ By $\gcd (m, n)=1$, we have $\gcd
(q, n)=1$. It follows that
$$q^{\frac 12 \alpha_q d(m)}  \mid
x_1^{(q-1, n)}-y_1^{(q-1, n)}.$$ Noting that $(q-1, n) \mid q-1$,
we have
$$q^{\frac 12 \alpha_q d(m)}\mid
x_1^{q-1}-y_1^{q-1}.$$

{\bf Case 2:} $q\nmid d_1$ and $q\nmid z$. Then
$\delta_q=\beta_q=0$. By \eqref{a1},
$$q^{\frac 12 \alpha_q d(m)} \mid x_1^n-y_1^n.$$
Similar to Case 1,  we have
$$q^{\frac 12 \alpha_q d(m) }\mid
x_1^{q-1}-y_1^{q-1}.$$

{\bf Case 3:} $q\nmid d_1$ and $q\mid z$.

By \eqref{a1},
$$q^{\delta_q-1} q^{\frac 12 \alpha_q d(m)} \mid q^{\delta_q} (x_1^n-y_1^n).$$
It follows that
$$q^{\frac 12 \alpha_q d(m) -1} \mid x_1^n-y_1^n.$$
If $\frac 12 \alpha_q d(m) -1\ge 1$, then, similar to Case 1, we
have \begin{equation}\label{q1}q^{\frac 12 \alpha_q d(m) -1}\mid
x_1^{q-1}-y_1^{q-1}.\end{equation} It is clear that \eqref{q1}
also holds if $\frac 12 \alpha_q d(m) -1=0$.

This completes the proof of Lemma \ref{lem3a}.
\end{proof}

\begin{remark} In 1943, Guderson \cite{Guderson1943} proved that, if $a>b\ge 1$ are two integers and $n$ is a positive
integer, then $n^2 (\text{rad} (n))^{-1}\mid \phi (a^n-b^n)$,
where $\text{rad} (n)$ is the radical of $n$, i.e, the product of
all distinct prime divisors of $n$. In 1961, Rotkiewicz
\cite{Rotkiewicz1961} proved that $n^{d(n)/2}\mid \phi (a^n-b^n)$,
where $d(n)$ is the number of positive divisors of
$n$.\end{remark}

\begin{lemma}\label{lem3} Let ${\cal Q}_m$ be the set of all prime
divisors of $m$ and let $d(m)$ be defined as in Lemma \ref{lem3a}.
Then
$$d(m)< 2\max \{ p(m), x\} $$
and
$$|{\cal Q}_m|<\frac {\log ( 2\max \{ p(m), x\} )}{\log 2}.$$
\end{lemma}

\begin{proof} By Lemma \ref{lem1}, $p(m)\ge 3$.

If $p(m)\mid z$, then, by $p(m) \mid m$ and Lemma \ref{lem3a}, we
have
$$p(m)^{\frac 12 \alpha_{p(m)} d(m) -1}\mid
x_1^{p(m)-1}-y_1^{p(m)-1}.$$ Noting that $1\le z\le x-y<x$, for
$p(m) \mid z$, we have
\begin{eqnarray*}p(m)^{\frac 12 d(m)-1}&\le &
p(m)^{\frac 12 \alpha_{p(m)} d(m)-1}\\
&\le & x_1^{p(m)-1}-y_1^{p(m)-1}\\
&<&x_1^{p(m)-1}\\
&\le &x^{p(m)-1}\\
&< &z^{-1}x^{p(m)}\\
&\le &p(m)^{-1}x^{p(m)}.\end{eqnarray*} So
\begin{equation*}d(m)<\frac{2p(m)}{\log p(m)}\log
x.\end{equation*}

If $p(m)\nmid z$, then, by $p(m) \mid m$ and Lemma \ref{lem3a}, we
have
$$p(m)^{\frac 12 \alpha_{p(m)} d(m) }\mid
x_1^{p(m)-1}-y_1^{p(m)-1}.$$ Hence \begin{eqnarray*}p(m)^{\frac 12
d(m)}&\le &
p(m)^{\frac 12 \alpha_{p(m)} d(m)}\\
&\le & x_1^{p(m)-1}-y_1^{p(m)-1}\\
&<&x^{p(m)}.\end{eqnarray*} So
\begin{equation*}d(m)<\frac{2p(m)}{\log p(m)}\log
x.\end{equation*}

In any way, we have
\begin{equation}\label{eq1}d(m)<\frac{2p(m)}{\log p(m)}\log x.\end{equation}

If $p(m)\ge x$, then $\log p(m)\ge \log x$. It follows from
\eqref{eq1} that
$$d(m)< 2p(m).$$
If $p(m)< x$, then, by $p(m)\ge 3$,
$$\frac{2p(m)}{\log p(m)}<\frac{2x}{\log x}.$$
It follows from \eqref{eq1} that
$$d(m)< 2x.$$
In any way, we have
$$d(m)< 2\max \{ p(m), x\}.$$
Noting that
$$2^{|{\cal Q}_m|}\le d(m),$$
we have
$$|{\cal Q}_m|<\frac {\log ( 2\max \{ p(m), x\} )}{\log 2}.$$
This completes the proof of Lemma \ref{lem3}.
\end{proof}

\begin{lemma}(\cite[Lemma 1.2]{Luca2015})\label{lem3ab} For $N\ge 3$, we have
$$\frac N{\phi (N)} \le 1.79\log\log N+\frac{2.5}{\log\log N}.$$
\end{lemma}

\begin{lemma}\label{lem2.1} Let $(x, y, z, m, n)$ be a
nontrivial solution of the  equation \eqref{eqn1a} in positive
integers $x, y, z, m, n$ with $1\le z\le x-y$ and $\gcd (m, n)=1$.
Then $$x<\prod_{p\mid z(x^m-y^m)/(x-y)} \left( 1+\frac
1{p-1}\right) .$$
\end{lemma}

\begin{proof} We follow the proof of \cite[Lemma 2.1]{Luca2015}.
Since $(x, y, z, m, n)$ is a nontrivial solution of the  equation
\eqref{eqn1a}, it follows that $m>n$. By \eqref{eqn1a},
\begin{eqnarray*}x&\le & x^{m-n}\\
&<&\frac{z(x^m-y^m)/(x-y)}{z(x^n-y^n)/(x-y)}\\
&=& \frac{z(x^m-y^m)/(x-y)}{\phi (z(x^m-y^m)/(x-y))}\\
&=& \prod_{p\mid z(x^m-y^m)/(x-y)} \left( 1-\frac 1p
\right)^{-1}\\
&=& \prod_{p\mid z(x^m-y^m)/(x-y)} \left( 1+\frac 1{p-1}\right) .
\end{eqnarray*}
This completes the proof of Lemma \ref{lem2.1}.
\end{proof}

\begin{lemma}\label{lem2.1a} Let $d$ be a divisor of $m$ with
$d\ge 40$ and let
$$S_d=\sum_{\ell_p=d} \frac 1p.$$
Then
$$S_d<\frac{1}{4d}+\frac1{d\log (d+1)}+\frac{2\log\log d}{\phi
(d)} +\frac{2\log\log x}{\phi (d)\log d}.$$
\end{lemma}

\begin{proof} We follow the proof of \cite[Lemma 2.1]{Luca2015}.
For the convenience of the reader, we give the details here.
Recall that, for $p\nmid x_1y_1$,  $\ell_p$ is the least positive
integer $\ell $ such that $p\mid x_1^\ell -y_1^\ell $. By Fermat's
theorem, $p\mid x_1^{p-1} -y_1^{p-1} $ for $p\nmid x_1y_1$. It
follows that $\ell_p \mid p-1$. Let
$${\cal P}_d =\{ p : \ell_p=d\} .$$
Then $d\mid p-1$ for all $p\in {\cal P}_d$. Hence
$$(d+1)^{|{\cal P}_d|} \le \prod_{p\in {\cal P}_d} p \le
x_1^d-y_1^d<x_1^d\le x^d.$$ It follows that
\begin{equation}\label{ee1}|{\cal P}_d|\le \frac{d\log x}{\log
(d+1)}.\end{equation} Let $\pi (X; d, 1)$ denote the number of
primes $p\le X$ with $d\mid p-1$. By the Brun-Titchmarsh theorem
due to Montgomery and Vaughan \cite{Montgomery1973},
$$\pi (X; d, 1)<\frac{2X}{\phi (d) \log (X/d)} \quad \text{ for
all } X>d\ge 2.$$ Let $A_d=\{ p\le 4d : d\mid p-1\} $. We split
$S_d$ as follows:
\begin{eqnarray*}S_d &= & \sum_{\substack{p\le 4d\\ \ell_p=d}} \frac 1p +\sum_{\substack{ 4d< p\le d^2\log x\\ \ell_p=d}}\frac 1p
+\sum_{\substack{ p>d^2\log x \\ \ell_p=d}} \frac 1p\\
 &\le & \sum_{p\in A_d} \frac 1p +\sum_{\substack{ 4d< p\le d^2\log x\\ d\mid p-1}}\frac 1p
+\sum_{\substack{ p>d^2\log x \\ p\in {\cal P}_d}} \frac 1p\\
&:=&T_1+T_2+T_3.\end{eqnarray*} For $T_2$, we have
\begin{eqnarray*}T_2&=&\int_{4d}^{d^2\log x} \frac 1t d \pi (t;
d, 1)\\
&=&\frac {\pi (t; d, 1)}t \Big|_{t=4d}^{d^2\log x}
+\int_{4d}^{d^2\log x} \frac {\pi (t; d, 1)}{t^2} d t\\
&\le & \frac{2}{\phi (d) \log (d\log x)} -\frac{\pi (4d; d,
1)}{4d}+\frac 2{\phi (d)}\int_{4d}^{d^2\log x} \frac 1{t\log (t/d)} d t\\
&=& \frac{2\log\log (d\log x)}{\phi (d)} -\frac{\pi (4d; d,
1)}{4d} +\frac 2{\phi (d)} \left( \frac{1}{\log (d\log x)}
-\log\log 4 \right).
\end{eqnarray*}
Since $d\ge 40$ and $x\ge x_1\ge 3$, it follows that
$$\frac{1}{\log (d\log x)}-\log\log 4\le \frac{1}{\log 40}-\log\log
4<0.$$ Hence
$$T_1+T_2\le \frac{2\log\log (d\log x)}{\phi (d)} -\frac{\pi (4d; d,
1)}{4d}+\sum_{p\in A_d} \frac 1p.$$ By Lemma \ref{lem1}, $2\nmid
m$. So $d$ is odd. Thus $$A_d\subseteq \{ 2d+1 \},\quad \pi (4d;
d, 1)=|A_d|\le 1.$$ It follows that
$$-\frac{\pi (4d; d,
1)}{4d}+\sum_{p\in A_d} \frac 1p \le -\frac 1{4d} +\frac 1{2d+1} <
\frac{1}{4d}.$$ So
$$T_1+T_2 <\frac{1}{4d}+\frac{2\log\log (d\log x)}{\phi
(d)}.$$ For $T_3$, by \eqref{ee1},
$$T_3<\frac{|{\cal P}_d|}{d^2\log x}<\frac1{d\log (d+1)}.$$
Therefore,
$$S_d<\frac{1}{4d}+\frac1{d\log (d+1)}+\frac{2\log\log (d\log x)}{\phi
(d)}.$$ Noting that
\begin{eqnarray*}\log\log (d\log x)&=&\log (\log d+\log\log x)\\
&=&\log\log d +\log\left( 1+\frac{\log\log x}{\log d}\right)\\
&<& \log\log d +\frac{\log\log x}{\log d},\end{eqnarray*} we have
$$S_d<\frac{1}{4d}+\frac1{d\log (d+1)}+\frac{2\log\log d}{\phi
(d)} +\frac{2\log\log x}{\phi (d)\log d}.$$ This completes the
proof of Lemma \ref{lem2.1a}.
\end{proof}

\begin{lemma}\label{lem2.1aa} The function $f(x)=\log\log x$ is
sub-multiplicative on $[78, +\infty )$, that is, for any $x_1,
x_2\ge 78$, we have
$$\log\log (x_1x_2)\le (\log\log x_1)(\log\log x_2).$$
\end{lemma}

\begin{proof}
  Suppose that  $x_1, x_2\ge 78$. Then
$$\frac 1{\log x_1} + \frac 1{\log x_2}<0.4591.$$ It follows that
$$\log x_1 +\log x_2< 0.4591 (\log x_1)(\log x_2).$$
Therefore,
\begin{eqnarray*}\log\log (x_1 x_2)&=&\log (\log x_1 +\log x_2)\\
&<&\log \left( 0.4591 (\log x_1)(\log x_2)\right)\\
&=& \log\log x_1+\log\log x_2 +\log 0.4591\\
&=& (\log\log x_1)(\log\log x_2)\\
&& -(\log\log x_1-1)(\log\log x_2-1)+1+\log 0.4591\\
&<& (\log\log x_1)(\log\log x_2).
\end{eqnarray*}
The last inequality holds since
\begin{eqnarray*}&&-(\log\log x_1-1)(\log\log x_2-1)+1+\log
0.4591\\
&\le&  -(\log\log 78-1)^2+1+\log 0.4591<0.\end{eqnarray*} This
completes the proof of Lemma \ref{lem2.1aa}.
\end{proof}

\section{Proof of Theorem \ref{thm3} for $x>80$}\label{sec2}

In this section, we always assume that $(x, y, z, m, n)$ is a
nontrivial solution of the  equation \eqref{eqn1a} in positive
integers $x, y, z, m, n$ with $x>80$, $1\le z\le x-y$ and $\gcd
(m, n)=1$. Then $m>n\ge 1$. Recall that
$$\gcd (x, y)=d_1,\quad x=x_1d_1,\quad y=y_1d_1.$$
Then $x_1$ and $y_1$ are both odd (see Section \ref{sec1}).
  For each prime $p\ge 3$ with $p\nmid x_1y_1$, let $\ell_p$ be the least positive integer $\ell $
  such that $p\mid x_1^\ell -y_1^\ell $.
Then $p\mid x_1^m-y_1^m$ if and only if $\ell_p \mid m$. Let
$p(m)$ be the least prime divisor of $m$.

 We divide into two subsections:
$p(m)\le x$ and $p(m)> x$.

\subsection{$p(m)\le
x$}\label{subsec2}

In this subsection, we always assume that $p(m)\le x$ and $x>80$.

 By Lemma \ref{lem2.1}, we
have
$$x<\prod_{p\mid z(x^m-y^m)/(x-y)} \left( 1+\frac 1{p-1}\right) .$$
It follows that
\begin{eqnarray*}\log x &<&\sum_{p\mid z(x^m-y^m)/(x-y)}\log \left( 1+\frac
1{p-1}\right)\\
&\le &\log \frac{15}{4}+\sum_{\substack{p\mid z(x^m-y^m)/(x-y)\\
p\ge 7}}\log \left( 1+\frac
1{p-1}\right)\\
&\le &\log \frac{15}{4}+\sum_{\substack{p\mid z(x^m-y^m)/(x-y)\\
p\ge 7}}
\frac 1{p-1}\\
&\le &\log \frac{15}{4}+\sum_{p\ge 7} \frac 1{p(p-1)}+\sum_{\substack{p\mid z(x^m-y^m)/(x-y)\\
p\ge 7}} \frac 1{p}\\
&<&\log \frac{15}{4}+\sum_{7\le p<547} \frac
1{p(p-1)}+\frac1{546}+\sum_{\substack{p\mid z(x^m-y^m)/(x-y)\\
p\ge 7}} \frac
1{p}\\
&<& 1.38+\sum_{\substack{p\mid z(x^m-y^m)/(x-y)\\
p\ge 7}} \frac 1{p}\\
&\le & 1.38+\sum_{7\le p\le x^4} \frac 1{p} +\sum_{\substack{p\mid
z(x^m-y^m)/(x-y)\\ p>x^4}} \frac 1{p} .
\end{eqnarray*}
By \cite{Rosser1962}, for $t\ge 286$,
$$\sum_{p\le t}\frac 1p<\log\log t+0.2615+\frac 1{2\log^2 t}<\log\log
t+0.2772.$$ It follows that
\begin{eqnarray*}\sum_{7\le p\le x^4} \frac 1{p}&<&\log\log x^4+0.2772-\frac 12-\frac 13-\frac
15\\
&<&\log\log x+0.6302.
\end{eqnarray*}
Hence
\begin{eqnarray*}\log x &<& \log\log x+1.38+0.6302+\sum_{\substack{p\mid z(x^m-y^m)/(x-y)\\
p> x^4}} \frac 1{p}\\
&<& \log\log x+2.011+\sum_{\substack{p\mid z(x^m-y^m)/(x-y)\\
p>x^4}} \frac 1{p}.\end{eqnarray*} It is clear that, if $p>x^4$,
then $p\nmid x_1y_1$. Recall that $\ell_p$ is the least positive
integer $\ell$ with $p\mid x_1^\ell -y_1^\ell$. If $p>x^4$, then
$\ell_p\ge 5$. Otherwise, $\ell_p\le 4$ and $p\le
x_1^{\ell_p}-y_1^{\ell_p}<x_1^4\le x^4$, a contradiction. If
$p>x^4$, then, by $x=d_1x_1$ and $1\le z\le x-y$, we have $p\nmid
d_1 z$. Hence, if $p\mid z(x^m-y^m)/(x-y)$ and $p>x^4$, then
$p\mid x_1^m -y_1^m$. It follows that $\ell_p \mid m$. Thus
$$\sum_{\substack{p\mid z(x^m-y^m)/(x-y)\\ p>x^4}} \frac 1{p} = \sum_{\substack{d\mid
m\\ d\ge 5}} T_d,$$ where
$$T_d=\sum_{\substack{\ell_p =d\\
p>x^4}} \frac 1{p} .$$ Let
$${\cal P}'_d=\{ p : \ell_p=d, p>x^4\} .$$
Then
$$x^{4|{\cal P}'_d|}<\prod_{p\in {\cal P}'_d} p\le x_1^d-y_1^d<x^d.$$
It follows that
$$|{\cal P}'_d|<\frac 14 d.$$
If $d\le x^2$, then
$$T_d\le \frac{|{\cal P}'_d|}{x^4}<\frac{d}{4x^4}\le \frac
1{4x^2}.$$ Thus
$$\sum_{\substack{d\mid
m \\ 5\le d\le x^2}} T_d <x^2\frac 1{4x^2}=0.25.$$

Now we estimate
$$\sum_{\substack{d\mid
m\\ d>x^2}} T_d.$$

By Lemma \ref{lem2.1a}, for $d>x^2>80^2$, we have
$$T_d\le S_d<\frac{1}{4d}+\frac1{d\log (d+1)}+\frac{2\log\log d}{\phi
(d)} +\frac{2\log\log x}{\phi (d)\log d}.$$ By Lemma \ref{lem3ab},
$$\frac d{\phi (d)} \le 1.79\log\log d+\frac{2.5}{\log\log d}.$$
It follows that
\begin{eqnarray}\label{eqn2}T_d&<&\frac{1}{4d}+\frac 1{d\log (d+1)} +
\frac{3.58(\log\log d)^2}{d}\nonumber \\
&& + \frac{3.58\log\log d}{d\log d} \log\log x +\frac 5d +\frac{5
\log\log x}{d(\log d)(\log\log d)}. \end{eqnarray}
  For
$d>x^2$, by \eqref{eqn2}, we have
\begin{eqnarray*}T_d&<&\frac{1}{4x^2}+\frac 1{x^2\log (x^2+1)} +
\frac{3.58(\log\log x^2)^2}{x^2}\nonumber \\
&& + \frac{3.58\log\log x^2}{x^2\log x^2} \log\log x +\frac 5{x^2}
+\frac{5 \log\log x}{x^2(\log x^2)(\log\log x^2)} .
\end{eqnarray*}
By Lemma \ref{lem3} and $p(m)\le x$, we have $d(m)<2x$. Hence
$$|\{ d : d\mid m, d>x^2 \} |< d(m)<2x.$$ Hence, by $x>80$, we
have
\begin{eqnarray*}\sum_{\substack{d\mid m\\ d>x^2}}T_d&<&\frac{1}{2x}+\frac 2{x\log (x^2+1)} +
\frac{7.16(\log\log x^2)^2}{x}\nonumber \\
&& +\frac{7.16\log\log x^2}{x\log x^2} \log\log x +\frac {10}{x}
+\frac{10 \log\log x}{x(\log x^2)(\log\log x^2)} \\
&<& 0.6.
\end{eqnarray*}
Therefore,
\begin{eqnarray*}\log x &<&\log\log x+2.011+\sum_{\substack{p\mid z(x^m-y^m)/(x-y)\\
p>x^4}} \frac 1{p}\\
&=&\log\log x+2.011+\sum_{\substack{d\mid m\\ d\ge 5}} T_d\\
&\le &\log\log x+2.011+\sum_{\substack{d\mid
m \\ 5\le d\le x^2}} T_d+\sum_{\substack{d\mid m\\ d>x^2}}T_d\\
&\le &\log\log x+2.011+0.25+0.6\\
&= &\log\log x+2.861.
\end{eqnarray*}
Since $x>80$, it follows that $$\log x -\log\log x >\log 80
-\log\log 80 >2.9, $$  a contradiction.

\subsection{$p(m)> x$}\label{subsec3}

In this subsection, we always assume that $p(m)> x>80$. Since
$p(m)$ is a prime, it follows that $p(m)\ge 83$.

 By Lemma \ref{lem2.1}, we
have
$$x<\prod_{p\mid z(x^m-y^m)/(x-y)} \left( 1+\frac 1{p-1}\right) .$$
Similar to the  arguments in the previous subsection, by $1\le
z\le x-y<x$, we have
\begin{eqnarray*}\log x &<& 1.38+\sum_{\substack{p\mid z(x^m-y^m)/(x-y)\\
p\ge 7}} \frac 1{p}\\
&<& 1.38+\sum_{7\le p\le  x} \frac 1{p}+\sum_{\substack{p\mid (x^m-y^m)/(x-y)\\
p>x}} \frac 1{p}.
\end{eqnarray*}
By \cite{Rosser1962}, for $t\ge 286$,
$$\sum_{p \le t}\frac 1p<\log\log t+0.2615+\frac 1{2\log^2 t}<\log\log t+0.2772.$$
A simple calculation shows that, for $80\le t\le 286$,
$$\sum_{p \le t}\frac 1p <\log\log t+0.2965.$$
It follows from $x>80$ that
$$\sum_{7\le p\le  x} \frac 1{p}<\log\log x+0.2965-\frac 12-\frac 13-\frac
15<\log\log x-0.7368.$$ It is clear that, if $p>x$, then $p\nmid
x_1y_1$, and by the definition of $\ell_p$, we have $\ell_p>1$.
Hence
\begin{eqnarray*}
\log x &<& 1.38+\log\log x-0.7368+\sum_{\substack{p\mid (x^m-y^m)/(x-y)\\
p>x}} \frac 1{p}\\
&=& \log\log x+0.6432 +\sum_{\substack{p\mid (x^m-y^m)/(x-y)\\
p>x}} \frac 1{p}\\
&\le & \log\log x+0.6432 +\sum_{\substack{d\mid m\\ d>1 }}
S_d,\end{eqnarray*}   where
$$S_d=\sum_{\ell_p =d} \frac 1{p} .$$
 For $d\mid m$ and $d>1$, we have $d\ge p(m)>x>
80$. By Lemma \ref{lem2.1a},  we have
\begin{eqnarray*}S_d&<&\frac{1}{4d}+\frac1{d\log (d+1)}+\frac{2\log\log d}{\phi
(d)} +\frac{2\log\log x}{\phi (d)\log d}\\
&<&\frac{1}{4d}+\frac1{d\log x}+\frac{2\log\log d}{\phi
(d)} +\frac{2\log\log x}{\phi (d)\log x}\\
&<&\frac{1}{4\phi (d)}+\frac1{\phi (d)\log 80}+\frac{2\log\log
d}{\phi
(d)} +\frac{2\log\log 80}{\phi (d)\log 80}\\
&<&\frac{1.2}{\phi (d)} +\frac{2\log\log d}{\phi (d)}\\
&<&\frac{3\log\log d}{\phi (d)} .
\end{eqnarray*} It follows that
\begin{equation*}\sum_{\substack{d\mid
m\\ d>1}} S_d<  3\sum_{\substack{d\mid m\\
d>1}}\frac{\log\log d}{\phi (d)} .\end{equation*}

By Lemma \ref{lem3} and $p(m)>x$,
$$|{\cal Q}_m|<\frac {\log ( 2p(m) )}{\log 2}  :=g(m),$$
where ${\cal Q}_m$ is the set of all prime divisors of $m$.

 Since
$p(m)\ge 83$, it follows from Lemma \ref{lem2.1aa} that
\begin{eqnarray*}\sum_{\substack{d\mid m\\
d>1}}\frac{\log\log d}{\phi (d)} &<&\prod_{q\in {\cal
Q}_m} \left( 1 +\frac {\log\log q} {\phi (q)} +\frac {\log\log q^2}{\phi (q^2)} + \cdots \right) -1 \\
&\le &\prod_{q\in {\cal
Q}_m} \left( 1 +\frac {\log\log q} {\phi (q)} +\frac {(\log\log q)^2}{\phi (q^2)} + \cdots \right) -1 \\
&=& \prod_{q\in {\cal
Q}_m}  \left( 1 +\frac {\log\log q} {q-1} \frac 1{1-(\log\log q)/q}  \right) -1 \\
&\le & \left( 1 +\frac {\log\log p(m)} {p(m)-1} \frac 1{1-(\log\log p(m))/p(m)}  \right)^{g(m)} -1 \\
&:=& T-1.
\end{eqnarray*}
By $p(m)\ge 83$,
\begin{eqnarray*}\log T &=&g(m)\log \left( 1 +\frac {\log\log p(m)} {p(m)-1} \frac 1{1-(\log\log p(m))/p(m)}
\right)\\
&<& g(m)\frac {\log\log p(m)} {p(m)-1} \frac 1{1-(\log\log
p(m))/p(m)}\\
&=&  \frac {\log (2p(m))}{\log 2} \frac {\log\log p(m)} {p(m)-1}
\frac {p(m)}{p(m)-\log\log
p(m)} \\
&=&\frac 1{\log 2} \frac {\log (2p(m))}{\sqrt{p(m)}}
\frac{\log\log p(m)} {\sqrt{p(m)}} \frac{p(m)}{p(m)-1} \frac
{p(m)}{p(m)-\log\log
p(m)} \\
&\le &\frac 1{\log 2} \frac {\log 166}{\sqrt{83}} \frac{\log\log
83} {\sqrt{83}} \frac{83}{83-1} \frac {83}{83-\log\log
83} \\
&<& 0.137
\end{eqnarray*}
It follows that
$$\sum_{\substack{d\mid m\\
d>1}}\frac{\log\log d}{\phi (d)} < e^{0.137} -1<0.147.$$ By $x>80$
we have
\begin{eqnarray*}\sum_{\substack{d\mid
m\\ d>1}} S_d< 3\sum_{\substack{d\mid m\\
d>1}}\frac{\log\log d}{\phi (d)}< 3\times 0.147
 < 0.45.\end{eqnarray*} It follows that
\begin{eqnarray*}\log x &<&\log\log x+0.6432 +\sum_{\substack{d\mid m\\ d>1 }}
S_d\\
&<&\log\log x+0.6432 +0.45\\
&<&\log\log x+1.1.
\end{eqnarray*}
Since $x>80$, it follows that $$\log x -\log\log x >\log 80
-\log\log 80 >2.9, $$  a contradiction.

\section{Proof of Theorem \ref{thm3} for $x\le 80$}\label{sec3}

In this section, we always assume that $(x, y, z, m, n)$ is a
nontrivial solution of the  equation \eqref{eqn1a} in positive
integers $x, y, z, m, n$ with $x\le 80$, $1\le z\le x-y$ and $\gcd
(m, n)=1$. Then $m>n\ge 1$. Recall that
$$\gcd (x, y)=d_1,\quad x=x_1d_1,\quad y=y_1d_1.$$
Then $x_1$ and $y_1$ are both odd (see Section \ref{sec1}).
  For each prime $p\ge 3$ with $p\nmid x_1y_1$, let $\ell_p$ be the least positive integer $\ell $
  such that $p\mid x_1^\ell -y_1^\ell $.
Then $p\mid x_1^m-y_1^m$ if and only if $\ell_p \mid m$. Let
$p(m)$ be the least prime divisor of $m$.
 Now \eqref{eqn1a}
becomes
\begin{equation}\label{eqw1<80a}\phi \left(
zd_1^{m-1}\frac{x_1^m-y_1^m}{x_1-y_1}\right)
=zd_1^{n-1}\frac{x_1^n-y_1^n}{x_1-y_1}.\end{equation}

First, we give the following lemmas.

\begin{lemma}\label{lem1a} If $q$ is a prime factor of $m$, then $q\nmid x_1-y_1$. \end{lemma}

\begin{proof} Suppose that $q$ is a prime with $q\mid m$ and $q\mid
x_1-y_1$. We will derive a contradiction. By Lemma \ref{lem1},
$q\ge 3$. Let
$$m=q^\alpha m_1,\quad x_1-y_1=q^\beta w_0,\quad z=q^\mu z_1,\quad d_1=q^\delta
t_1,$$ $$
 \alpha , \beta , m_1, w_0, z_1, t_1\in \mathbb{Z}^+, \ q\nmid m_1w_0z_1t_1,\ \mu , \delta \in\mathbb{Z}_{\ge 0}.$$
By induction on $k$,
$$x_1^{q^k}=y_1^{q^k}+q^{k+\beta } w_k, \quad w_k\in \mathbb{Z}.$$
It follows that
$$x_1^{q^\alpha}-y_1^{q^\alpha}=q^{ \alpha +\beta } w_\alpha , \quad w_\alpha \in \mathbb{Z}.$$
So
\begin{equation}\label{ac1} q^\alpha \mid \frac{x_1^{q^\alpha}-y_1^{q^\alpha}}{x_1-y_1} .\end{equation}
In view of  $q^\alpha \mid m$ and \eqref{ac1}, we have
$$q^{\alpha } \mid \frac{x_1^m-y_1^m}{x_1-y_1}.$$
By Carmichael's primitive divisor theorem (see
\cite{Carmichael1913}), the integer $x_1^q-y_1^q$
 has a primitive prime divisor $p\equiv 1\pmod{q}$. Since $q\mid
 m$, it follows that $$\frac{x_1^q-y_1^q}{x_1-y_1} \mid \frac{x_1^m-y_1^m}{x_1-y_1}.$$ So $$p\mid
 \frac{x_1^m-y_1^m}{x_1-y_1}.$$
Thus $$q^{\alpha  } p\mid \frac{x_1^m-y_1^m}{x_1-y_1}.$$ So
$$q^{\alpha  +\mu +\delta (m-1) } p\mid zd_1^{m-1}\frac{x_1^m-y_1^m}{x_1-y_1}.$$
Hence
$$q^{\alpha  +\mu +\delta (m-1)-1 }(q-1)(p-1) \mid \phi \left( zd_1^{m-1}\frac{x_1^m-y_1^m}{x_1-y_1}\right) .$$
Since $q\mid p-1$, it follows  that
\begin{equation}\label{eqn3} q^{\alpha  +\mu +\delta (m-1) }
\mid \phi \left( zd_1^{m-1}\frac{x_1^m-y_1^m}{x_1-y_1}\right)
.\end{equation} By \eqref{eqw1<80a} and \eqref{eqn3},
\begin{equation*} q^{\alpha  +\mu +\delta (m-1) } \mid zd_1^{n-1}\frac{x_1^n-y_1^n}{x_1-y_1}.\end{equation*}
So
\begin{equation}\label{a2} q^{\alpha  +\mu +\delta (m-1) } \mid q^{\mu +\delta (n-1) } \frac{x_1^n-y_1^n}{x_1-y_1}.\end{equation}
By $\gcd (m, n)=1$, we have $q\nmid n$. It follows from
$x_1-y_1=q^\beta x_0$ $(q\nmid x_0, \beta \ge 1)$ that
$$q\nmid \frac{x_1^n-y_1^n}{x_1-y_1}.$$
In view of \eqref{a2},
$$\alpha  +\mu +\delta (m-1)\le \mu +\delta (n-1).$$
Noting that $m>n$, we have $\alpha  =0$, a contradiction with
$\alpha  >0$. Therefore, $q\nmid x_1-y_1$.
\end{proof}

\begin{lemma}\label{lem4} We have
\begin{eqnarray*}x_1\frac{\phi (zd_1)}{z}< \prod_{\substack{\ell_p\mid m
\\ \ell_p
>1}} \left( 1+\frac 1{p-1}\right) .\end{eqnarray*}\end{lemma}

\begin{proof} By Lemma \ref{lem2.1},  we
have
\begin{eqnarray*}x&<&\prod_{p\mid z(x^m-y^m)/(x-y)} \left( 1+\frac 1{p-1}\right) \\
&\le &\prod_{p\mid zd_1} \left( 1+\frac 1{p-1}\right) \prod_{p\mid (x_1^m-y_1^m)/(x_1-y_1)} \left( 1+\frac 1{p-1}\right) \\
&=&\frac{zd_1}{\phi (zd_1)}  \prod_{p\mid (x_1^m-y_1^m)/(x_1-y_1)}
\left( 1+\frac 1{p-1}\right).\end{eqnarray*} It follows from
$x=d_1x_1$ that
\begin{eqnarray}\label{x}x_1\frac{\phi (zd_1)}{z}<  \prod_{p\mid (x_1^m-y_1^m)/(x_1-y_1)} \left( 1+\frac 1{p-1}\right) .\end{eqnarray}

Now we prove that, if
\begin{equation}\label{eqn3xx}p\mid \frac{x_1^m-y_1^m}{x_1-y_1},\end{equation}
then $\ell_p \mid m$ and $\ell_p >1$.

 Suppose that \eqref{eqn3xx} holds. Then
$p\mid x_1^m-y_1^m$.  Noting that $\gcd (x_1, y_1)=1$, we have
$p\nmid x_1y_1$.  By the properties of $\ell_p$ and $p\mid
x_1^m-y_1^m$, we have $\ell_p \mid m$.
 Now we prove  $\ell_p
>1$ by a contradiction. Suppose that $\ell_p =1$. Then $p\mid
x_1-y_1$.  So
$$\frac{x_1^m-y_1^m}{x_1-y_1}=x_1^{m-1}+x_1^{m-2}y_1+\cdots +x_1
y_1^{m-2}+y_1^{m-1}\equiv my_1^{m-1}\pmod p.$$ It follows from
\eqref{eqn3xx} that $p\mid m$. This contradicts Lemma \ref{lem1a}
since $p\mid m$ and $p\mid x_1-y_1$.

Now we have proved that, if \eqref{eqn3xx} holds, then $\ell_p$
exists, $\ell_p \mid m$ and $\ell_p >1$. Lemma \ref{lem4} follows
from \eqref{x}.
\end{proof}

\begin{lemma}\label{lem4a} If $x\le 80$ and $d$ is a divisor
of $m$ with $d\ge 173$, then
$$ \log \prod_{ \ell_p =d } \left( 1+\frac 1{p-1}\right) < \frac {3.3\log\log d}{\phi
(d)}.$$
\end{lemma}

\begin{proof} Suppose that $p$ is a prime with $\ell_p=d$. By
the definition of $\ell_p$, we have $p\nmid x_1y_1$. By Fermat's
theorem, $p\mid x_1^{p-1}-y_1^{p-1}$. It follows that $d\mid p-1$
when $\ell_p =d$. So $p\ge d+1$ when $\ell_p =d$. It is clear that
\begin{eqnarray*} \log \prod_{ \ell_p =d } \left( 1+\frac
1{p-1}\right) &=& \sum_{\ell_p =d} \log \left( 1+\frac
1{p-1}\right)\\
&<& \sum_{\ell_p =d} \frac 1{p-1}\\
&=& \sum_{\ell_p =d} \frac 1{p(p-1)} +\sum_{\ell_p =d} \frac 1{p}\\
&<& \sum_{n=d+1}^\infty \frac 1{n(n-1)} +\sum_{\ell_p =d} \frac 1{p}\\
&=& \frac 1d +S_d,
\end{eqnarray*}
where
$$S_d =\sum_{\ell_p =d} \frac 1{p}.$$
Since $d\ge  173$  and $x\ge x_1 \ge 3$, it follows from Lemma
\ref{lem2.1a} that
\begin{eqnarray*} S_d < \frac{1}{4d}+\frac 1{d\log (d+1)} +\frac {2\log\log d}{\phi
(d)} +\frac{2\log\log x}{\phi (d)\log d} .\end{eqnarray*} Noting
that
 $d\ge  173$ and $x\le 80$, we have
 \begin{eqnarray*}\log \prod_{ \ell_p =d } \left( 1+\frac 1{p-1}\right) &\le &
\frac{5}{4d}+\frac
1{d\log (d+1)} +\frac {2\log\log d}{\phi (d)} +\frac{2\log\log x}{\phi (d)\log d}\\
&\le & \frac{5}{4\phi (d)}+\frac 1{\phi (d)\log 174} +\frac
{2\log\log d}{\phi (d)} +\frac{2\log\log 80}{\phi (d)\log
173}\\
&<& \frac{2.1}{\phi (d)} +\frac {2\log\log d}{\phi (d)}\\
&<& \frac {3.3\log\log d}{\phi (d)}.\end{eqnarray*}
 This completes the
proof of Lemma \ref{lem4a}.
\end{proof}

\begin{lemma}\label{lem4ax} If $x\le 80$ and $d$ is a divisor
of $m$ with $p(d)\ge 173$, where $p(d)$ is the least prime divisor
of $d$, then
$$ \log \prod_{ \ell_p =d } \left( 1+\frac 1{p-1}\right) < 0.032.$$
\end{lemma}

\begin{proof} Let
$$d=p_1 p_2 \cdots p_t,$$
where $p_1\le p_2\le \cdots \le p_t$ are primes. It is easy to see
that
\begin{equation}\label{b1}\phi (p_1 p_2 \cdots p_t)\ge (p_1-1)\cdots
(p_t-1).\end{equation} By Lemma \ref{lem2.1aa} and $p(d)\ge 173$,
\begin{equation}\label{b2}\log\log d=\log\log  (p_1 p_2 \cdots p_t)\le (\log\log  p_1)\cdots
(\log\log  p_t).\end{equation} In view of Lemma \ref{lem4a},
\eqref{b1} and \eqref{b2},
\begin{eqnarray*}
\log \prod_{ \ell_p =d } \left( 1+\frac 1{p-1}\right) &<& \frac
{3.3\log\log d}{\phi (d)}\\
&\le & 3.3 \frac {\log\log p_1}{p_1-1} \cdots \frac {\log\log
p_t}{p_t-1}.
\end{eqnarray*}
Since $p_i\ge p(d)\ge 173$ $(1\le i\le t)$, it follows that
$$\frac {\log\log p_i}{p_i-1}\le \frac {\log\log 173}{172}<1,\quad 1\le i\le t.$$
Therefore,
$$\log \prod_{ \ell_p =d } \left( 1+\frac 1{p-1}\right)< 3.3 \frac {\log\log p_1}{p_1-1}\le 3.3\frac {\log\log
173}{172}<0.032.$$ This completes the proof of Lemma \ref{lem4ax}.
\end{proof}

\begin{lemma}\label{lem4b} If $q$ is a prime factor of $m$ with $q<173$, then $q^6\nmid x_1^{q-1}-y_1^{q-1}$. \end{lemma}

\begin{proof} By Lemma \ref{lem1a}, $q\nmid x_1-y_1$. Since $x\le
80$, it follows that $x_1\le 80$. By Lemma \ref{lem1}, $q\ge 3$. A
simple calculation by a computer shows that, for any integers
$1\le y_1<x_1\le 80$,  there are no odd primes $p<173$ such that
$$p\nmid x_1-y_1, \quad  p^6\mid x_1^{p-1}-y_1^{p-1}.$$
Hence $q^6\nmid x_1^{q-1}-y_1^{q-1}$. This completes the proof of
Lemma \ref{lem4b}.
\end{proof}

\begin{lemma}\label{lem4ba} If $q$ is a prime factor of $m$ and
$k$ is a positive integer such that
$$q^k \nmid x_1^{q-1}-y_1^{q-1},$$
then there are at most $k$ distinct primes $p$ with $\ell_{p} \mid
m$ and $q\mid \ell_p$.
\end{lemma}

\begin{proof} We prove the lemma by a contradiction.
Suppose that there are at least $k+1$ primes $p$ with $\ell_p \mid
m$ and $q\mid \ell_p$. Let $p_1, p_2, \dots , p_{k+1}$ be $k+1$
distinct primes with $\ell_{p_i} \mid m$ and $q\mid \ell_{p_i}$
$(1\le i\le k+1)$. Then, for $1\le i\le k+1$,
$$p_i\mid \frac{x_1^{\ell_{p_i} } - y_1^{\ell_{p_i} }}{ x_1-y_1
}$$ and
$$\frac{x_1^{\ell_{p_i} } - y_1^{\ell_{p_i} }}{ x_1-y_1
} \mid \frac{x_1^{m } - y_1^{m }}{ x_1-y_1 } .$$ It follows that
$$p_1p_2\cdots p_{k+1} \mid \frac{x_1^{m } - y_1^{m }}{ x_1-y_1 } .$$
Let
$$ z=q^\mu z_1,\quad d_1=q^\delta
t_1,$$ $$
 z_1, t_1\in \mathbb{Z}^+, \ q\nmid z_1t_1,\ \mu , \delta \in \mathbb{Z}_{\ge 0}.$$
Then
$$q^{\mu +(m-1)\delta }p_1p_2\cdots p_{k+1} \mid zd_1^{m-1}\frac{x_1^{m } - y_1^{m }}{ x_1-y_1 } .$$
It follows that
\begin{equation}\label{eqn3xxx}\phi (q^{\mu +(m-1)\delta }) (p_1-1)(p_2-1)\cdots (p_{k+1}-1)
\mid \phi \left( zd_1^{m-1}\frac{x_1^{m } - y_1^{m }}{ x_1-y_1
}\right) .\end{equation}By the definition of $\ell_p$, $p\nmid
x_1y_1$. Thus $p_i\nmid x_1y_1$ $(1\le i\le k+1)$.  By Fermat's
theorem, $p_i \mid x_1^{p_i-1}-y_1^{p_i-1}$. So $\ell_{p_i}\mid
p_i-1$ $(1\le i\le k+1)$. Since $q\mid \ell_{p_i}$, it follows
that  $q\mid p_i-1$ $(1\le i\le k+1)$. By \eqref{eqn3xxx},
$$q^{\mu +(m-1)\delta +k}\mid \phi \left( zd_1^{m-1}\frac{x_1^m-y_1^m}{x_1-y_1}\right) .$$
It follows from \eqref{eqw1<80a} that
$$q^{\mu +(m-1)\delta +k}\mid  zd_1^{n-1}\frac{x_1^n-y_1^n}{x_1-y_1} .$$
So
\begin{equation}\label{ew1}q^{\mu +(m-1)\delta +k} \mid q^{\mu +(n-1)\delta } \frac{x_1^n-y_1^n}{x_1-y_1}.\end{equation}
Noting that $m>n$, we have $\mu +(m-1)\delta\ge \mu +(n-1)\delta$.
It follows from \eqref{ew1} that
$$q^k \mid \frac{x_1^n-y_1^n}{x_1-y_1} .$$
So
$$q^k \mid x_1^n-y_1^n .$$
It follows from $\gcd (x_1, y_1)=1$ that $q\nmid x_1y_1$.   By
Euler's theorem,
$$q^k \mid x_1^{q^{k-1}(q-1)}-y_1^{q^{k-1}(q-1)} .$$
It follows that
\begin{equation}\label{aa1}q^k \mid x_1^{(n,q^{k-1}(q-1))}-y_1^{(n,q^{k-1}(q-1))} .\end{equation}
Since $\gcd (m, n)=1$ and $q\mid m$, we have $q\nmid n$. By
\eqref{aa1},
\begin{equation*}q^k \mid x_1^{(n, q-1)}-y_1^{(n, q-1)} .\end{equation*}
Noting that $(n, q-1) \mid q-1$, we have
$$q^k\mid x_1^{q-1}-y_1^{q-1} ,$$ a contradiction.
This completes the proof of Lemma \ref{lem4ba}.
\end{proof}

\begin{lemma}\label{lem4x} If $a$ and $b$ are two positive integers with $b\ge 2$, then $$\phi (ab)\ge \phi (2a).$$
 \end{lemma}

\begin{proof} If $2\mid a$, then $\phi (2a)=2\phi (a)$. If $2\nmid a$, then $\phi (2a)=\phi
(a)$. In any way, we have $2\phi (a)\ge \phi (2a)\ge \phi (a)$.

If $p$ is an odd prime and $p\nmid a$, then $$\phi (ap)=\phi (a)
(p-1)\ge 2\phi (a)\ge \phi (2a).$$

If $p$ is an odd prime and $p\mid a$, then $$\phi (ap)=p\phi (a)
> 2\phi (a)\ge \phi (2a).$$

In any way, if  $p$ is a prime, then $$\phi (ap)\ge \phi (2a)\ge
\phi (a).$$

Let $b=p_1p_2\cdots p_t$, where $p_1\le p_2\le \cdots \le p_t$ are
primes. Then
\begin{eqnarray*}\phi (ab)=\phi (ap_1p_2\cdots p_t)
\ge \phi (ap_1p_2\cdots p_{t-1}) \ge  \cdots \ge \phi (ap_1)\ge
\phi (2a).\end{eqnarray*} This completes the proof of Lemma
\ref{lem4x}.
\end{proof}

\begin{lemma}\label{lemu1} If $m$ has a prime divisor $q<173$, then  $d(m)\le 12$, where $d(m)$
is defined as in Lemma \ref{lem3a}. \end{lemma}

\begin{proof}
In view of Lemma \ref{lem4b}, $$q^6\nmid x_1^{q-1}-y_1^{q-1}.$$
 By Lemma \ref{lem3a},
\begin{equation}\label{eqn4x}q^{\frac 12 \alpha_q d(m) -1}\mid
x_1^{q-1}-y_1^{q-1},\end{equation} where $\alpha_q$ is defined as
in Lemma \ref{lem3a}. Hence
$$\frac 12 \alpha_q d(m) -1\le 5.$$
That is, $\alpha_q d(m)\le 12$. Therefore, $d(m)\le 12$. This
completes the proof of Lemma \ref{lemu1}.
\end{proof}

\begin{lemma}\label{lem4bcd} If $q$ is a prime factor of $m$ with $q<173$, then $q^3\nmid x_1^{q-1}-y_1^{q-1}$. \end{lemma}

\begin{proof} We prove the lemma by a contradiction. Suppose that
$$q^3\mid x_1^{q-1}-y_1^{q-1}.$$
By Lemma \ref{lem1a}, $q\nmid x_1-y_1$. Hence
\begin{equation}\label{eqn4xy}q\nmid x_1-y_1,\quad q^3\mid x_1^{q-1}-y_1^{q-1}.\end{equation}
It is clear that $x_1\le x\le 80$. A simple calculation by a
computer shows that,

(i) for $1\le y_1<x_1\le 9$, there is no  prime $\gamma <173$
satisfying \eqref{eqn4xy};

(ii)  for $1\le y_1<x_1$ and $10\le x_1\le 80$, there are at most
two primes $\gamma <173$ satisfying \eqref{eqn4xy}.

Since $q<173$, it follows from (i) and (ii) that $10\le x_1\le
80$.

 In order to derive a contradiction, we divide all positive divisors of $m$ into two classes:
$D_1$ is the set of all positive divisors of $m$ which have at
least one prime divisor $<173$ and $$D_2=\{ d : d\mid m, d>1,
d\notin D_1\}.$$  By Lemma \ref{lem4}, we have
\begin{eqnarray*}x_1\frac{\phi (zd_1)}{z}<
\prod_{\ell_p\in D_1} \left( 1+\frac 1{p-1}\right)
\prod_{\ell_p\in D_2} \left( 1+\frac 1{p-1}\right)
.\end{eqnarray*}

By Lemma \ref{lem1}, $2\nmid m$. It follows that, if $\ell_p>1$
and $\ell_p\mid m$, then, by $\ell_p\mid p-1$, we have $p\ge 7$.
Let $p_i$ be the $i$-th prime. If $\ell_{p_i}>1$ and
$\ell_{p_i}\mid m$, then $i\ge 4$.
 By Lemmas \ref{lem4b} and
\ref{lem4ba}, there are at most $6$ primes $p$ with $\ell_{p}\mid
m$ and $q\mid \ell_p$. So, for any given prime $q<173$,
\begin{equation}\label{abc1}|\{ p : \ell_p\in D_1, q\mid \ell_p\} | \le
6.\end{equation}
 By (ii) and \eqref{abc1}, $$|\{ p : \ell_p\in D_1 \} | \le 11.$$
It follows that
$$\prod_{\ell_p\in D_1} \left( 1+\frac 1{p-1}\right)\le \prod_{i=4}^{14} \left( 1+\frac
1{p_i-1}\right) <2.$$

By $q<173$, we have $q\notin D_2$. Noting that $1\notin D_2$, by
Lemma \ref{lemu1} we have  $|D_2|\le d(m)-2\le 10$.   By Lemma
\ref{lem4ax} and  $p(d)\ge 173$ for $d\in D_2$,
\begin{eqnarray*}\prod_{\ell_p\in D_2} \left( 1+\frac 1{p-1}\right)
&=& \prod_{d\in D_2} \exp \left(\log \prod_{ \ell_p =d } \left(
1+\frac 1{p-1}\right)\right)\\
& \le &  \prod_{d\in D_2} e^{0.032}\le e^{10\times 0.032}< 1.4
.\end{eqnarray*} Hence
\begin{eqnarray*}x_1\frac{\phi (zd_1)}{z}<
\prod_{\ell_p\in D_1} \left( 1+\frac 1{p-1}\right)
\prod_{\ell_p\in D_2} \left( 1+\frac 1{p-1}\right)<2\times
1.4<3.\end{eqnarray*}

It is clear that $z\le x-y<x\le 80$. Since $10\le x_1\le 80$ and
$x_1$ is odd, it follows that $x_1\ge 11$.

If $d_1\ge 2$, then, by Lemma \ref{lem4x}, we have $\phi (zd_1)
\ge \phi (2z)$. A simple calculation gives that,  for $1\le z\le
80$,
$$x_1\frac{\phi (zd_1)}{z} \ge 11 \frac{\phi (2z)}{z} >3,$$
a contradiction. Hence $d_1=1$. If $x_1\ge 13$, then, by  a simple
calculation, for $1\le z\le 80$,
$$x_1\frac{\phi (zd_1)}{z} \ge 13 \frac{\phi (z)}{z} >3,$$
a contradiction. So $x_1=11$.  Then $x=11$ and $z\le x-y\le 10$. A
simple calculation gives that,  for  $1\le z\le 10$,
$$x_1\frac{\phi (zd_1)}{z} = 11 \frac{\phi (z)}{z} >3,$$
a contradiction.

 This completes the proof
of Lemma \ref{lem4bcd}.
\end{proof}

\begin{lemma}\label{lem4cx} We have $p(m)\ge 173$. \end{lemma}

\begin{proof} Suppose that $p(m)< 173$.
We will derive a contradiction.
 Let $q$ be a prime divisor of $m$ with $q<173$. By Lemma
\ref{lem4bcd},
$$q^3\nmid x_1^{q-1}-y_1^{q-1} .$$
By Lemma \ref{lem3a},
\begin{equation}\label{eqn4x}q^{\frac 12 \alpha_q d(m) -1}\mid
x_1^{q-1}-y_1^{q-1},\end{equation} where $\alpha_q$ and $d(m)$ are
defined as in Lemma \ref{lem3a}. Hence
$$\frac 12 \alpha_q d(m) -1\le 2.$$
That is, $\alpha_q d(m)\le 6$. So $m\in \{ q, q^2, q\gamma ,
q\gamma^2 \} $.

 We divide all positive divisors of $m$ into two
classes: $D_1$ is the set of all positive divisors of $m$ which
have at least one prime divisor $<173$ and $$D_2=\{ d : d\mid m,
d>1, d\notin D_1\}.$$  By Lemma \ref{lem1}, $2\nmid m$. It follows
that, if $\ell_p>1$ and $\ell_p\mid m$, then, by $\ell_p\mid p-1$,
we have $p\ge 7$. Let $p_i$ be the $i$-th prime. If $\ell_{p_i}>1$
and $\ell_{p_i}\mid m$, then $i\ge 4$.

By Lemma \ref{lem4}, we have
\begin{eqnarray*}x_1\frac{\phi (zd_1)}{z}<
\prod_{\ell_p\in D_1} \left( 1+\frac 1{p-1}\right)
\prod_{\ell_p\in D_2} \left( 1+\frac 1{p-1}\right)
.\end{eqnarray*} Since $x_1$ and $y_1$ are odd and $x_1>y_1\ge 1$,
it follows that $x_1\ge 3$. So $x\ge 3$.

We divide into two cases:

{\bf Case 1:} $x>3$. Then either $x_1\ge 5$ or $d_1\ge 2$.

If $d_1\ge 2$, then,  by Lemma \ref{lem4x}, we have $\phi
(zd_1)\ge \phi (2z)$. It is clear that $z\le x-y<x\le 80$. A
simple calculation shows that, for $x_1\ge 3$, $1\le z\le 80$ and
$d_1\ge 2$,
$$x_1\frac{\phi (zd_1)}{z}\ge 3\frac{\phi (2z)}{z}>1.59.$$

If $d_1=1$ and $x_1=5$, then $x=x_1d_1=5$ and $z\le x-y\le 4$. It
is easy to see that, for $1\le z\le 4$,
$$x_1\frac{\phi (zd_1)}{z}= 5\frac{\phi (z)}{z}>1.59.$$

If $d_1=1$ and $x_1\ge 7$. A simple calculation shows that, for
$x_1\ge 7$ and $1\le z\le 80$,
$$x_1\frac{\phi (zd_1)}{z}\ge 7\frac{\phi (z)}{z}>1.59.$$

In any way, we have
$$x_1\frac{\phi (zd_1)}{z}>1.59.$$

Now we divide into two subcases:

{\bf Subcase 1.1:} All prime factors of $m$ are less than 173.
Then, $D_2=\emptyset$ and $D_1\subseteq \{ q, q^2, q\gamma ,
\gamma , \gamma^2 , q\gamma^2 \} $. By Lemmas
 \ref{lem4ba} and \ref{lem4bcd}, there are at most $3$ primes $p$ with $\ell_p \mid m$ and $q\mid \ell_p$
 and at most $3$ primes $p'$ with $\ell_{p'} \mid m$ and $\gamma \mid \ell_{p'}$.
It follows that
$$|\{ p : \ell_p\in D_1\} | \le 5.$$
So
$$\prod_{\ell_p\in D_1} \left( 1+\frac 1{p-1}\right)\le \prod_{i=4}^{8} \left( 1+\frac
1{p_i-1}\right) <1.56.$$ Thus $1.59<1.56$, a contradiction.

{\bf Subcase 1.2:} Only one of prime factors of $m$ is less than
173. Then $D_1\subseteq \{ q, q^2, q\gamma , q\gamma^2 \} $ and
$D_2\subseteq \{  \gamma , \gamma^2  \} $. By Lemmas
 \ref{lem4ba}  and \ref{lem4bcd}, there are at most $3$ primes $p$ with $\ell_p \mid m$ and $q\mid
 \ell_p$. It follows that
$$|\{ p : \ell_p\in D_1\} | \le 3.$$
So
$$\prod_{\ell_p\in D_1} \left( 1+\frac 1{p-1}\right)\le \prod_{i=4}^{6} \left( 1+\frac
1{p_i-1}\right)<1.4.$$  By Lemma \ref{lem4ax} and $p(d)\ge 173$
for $d\in D_2$,
\begin{eqnarray*}\prod_{\ell_p\in D_2} \left( 1+\frac 1{p-1}\right)
&=& \prod_{d\in D_2} \exp \left(\log \prod_{ \ell_p =d } \left(
1+\frac 1{p-1}\right)\right)\\
& \le &  \prod_{d\in D_2} e^{0.032}\le e^{2\times 0.032}< 1.1
.\end{eqnarray*}
 Thus $1.59<1.4\times 1.1=1.54$, a
contradiction.

{\bf Case 2:} $x=3$. Then $x_1=3$, $y_1=1$, $d_1=1$ and $z\le
x-y\le 2$. By $q\mid m$ and $m$ being odd (Lemma \ref{lem1}), we
have $q\nmid z$. By Lemma \ref{lem3a},
\begin{equation}\label{eqn4xx}q^{\frac 12 \alpha_q d(m)}\mid
3^{q-1}-1,\end{equation} where $\alpha_q$ and $d(m)$ are defined
as in Lemma \ref{lem3a}. A simple calculation shows that, there
are no odd primes $p <173$ with
$$p^3\mid 3^{p-1}-1.$$
So $\frac 12 \alpha_q d(m)\le 2$. It follows that $m\in \{ q,
q\gamma \} $.

Now we divide into two subcases:

{\bf Subcase 2.1:} $m=q\gamma $. Then $\frac 12 \alpha_q d(m)= 2$.
By \eqref{eqn4xx},
$$q^2\mid 3^{q-1}-1.$$
Since $q<173$, it follows from a simple calculation that $q=11$.
This implies that $\gamma \ge 173$, otherwise $\gamma =11$, a
contradiction. So $D_1=\{ 11, 11\gamma  \} $ and $D_2=\{ \gamma \}
$. Since $3^{11}-1=2\times 23\times 3851$, it follows that
$$\{ p : \ell_p=11 \} =\{ 23, 3851 \} .$$
By Carmichael's primitive divisor theorem (see
\cite{Carmichael1913}), the integer $3^{11\gamma}-1$
 has at least one primitive prime divisor $p'\equiv 1\pmod{11\gamma}$.
  By the definition of $\ell_{p'}$, we have $\ell_{p'}
=11\gamma $. Hence
$$| \{ p : \ell_p\in D_1\} |\ge 3.$$
Let $p'_1, p'_2, p'_3$ be three distinct primes with
$\ell_{p'_i}\in D_1$. Since
$$p'_i\mid \frac{3^{\ell_{p'_i}}-1}2,\quad \frac{3^{\ell_{p'_i}}-1}2
\mid \frac{3^m-1}2,$$ it follows that
$$p'_1p'_2p'_3\mid z\frac{3^m-1}2.$$
So
$$(p'_1-1)(p'_2-1)(p'_3-1)\mid \phi \left( z\frac{3^m-1}2 \right) .$$
Since $11\mid \ell_{p'_i}$ and $\ell_{p'_i} \mid p'_i-1$, it
follows that
$$11^3\mid \phi \left( z\frac{3^m-1}2 \right) .$$ In view of
\eqref{eqn1a},
$$11^3\mid z\frac{3^n-1}2.$$
By $z\le 2$, $11^3\mid 3^n-1$. Noting that $3^5-1=2\times 11^2$,
we have $11^2\mid 3^{(5,n)}-1$. It follows that $\gcd (5, n)=5$.
Let $n=5n_1$. Then
\begin{eqnarray*}3^n-1&=&3^{5n_1}-1
=(2\times 11^2 +1)^{n_1}-1\\
&=&\binom{n_1}{1} 2\times 11^2 +\binom{n_1}{2} (2\times
11^2)^2+\cdots +\binom{n_1}{n_1} (2\times 11^2)^{n_1}
.\end{eqnarray*} By $11^3\mid 3^n-1$, we have $11\mid n_1$ and
then $11\mid n$. Since $11=q\mid m$, it contradicts $\gcd (m,
n)=1$.

{\bf Subcase 2.2:} $m=q$. Then $D_1=\{ q\} $ and $D_2=\emptyset $.
By Lemmas
 \ref{lem4ba}  and \ref{lem4bcd},
$|\{ p : \ell_p =q \} |\le 3$. Noting that $x_1=3$ and $d_1=1$, we
have
\begin{eqnarray*}3\frac{\phi (z)}{z}
&<& \prod_{\ell_p\in D_1} \left( 1+\frac 1{p-1}\right)
\prod_{\ell_p\in D_2} \left( 1+\frac 1{p-1}\right)\\
&\le & \left( 1+\frac 1{2q}\right) \left( 1+\frac 1{4q}\right)\left( 1+\frac 1{6q}\right)\\
&\le & \left( 1+\frac 1{6}\right) \left( 1+\frac
1{12}\right)\left( 1+\frac 1{18}\right)\\
& <& 1.4 .\end{eqnarray*} But, for $z\in \{ 1, 2\} $,
$$3\frac{\phi (z)}{z}\ge 1.5,$$
a contradiction.

This completes the proof of Lemma \ref{lem4cx}.
\end{proof}

\medspace

{\bf Now we prove Theorem \ref{thm3} for $x\le 80$.}

 By Lemma \ref{lem4cx}, $p(m)\ge 173$. It follows that $p(m)>x$. For $d\mid m$ and $d>1$, we have $d\ge p(m)\ge 173$.

In view of  Lemma \ref{lem4} and Lemma \ref{lem4a}, we have
\begin{eqnarray}\label{xx1}\log \left( x_1\frac{\phi (zd_1)}{z}\right) &<&\log
\prod_{\substack{\ell_p >1\\ \ell_p \mid m}} \left( 1+\frac
1{p-1}\right)\nonumber\\
&=& \sum_{\substack{d\mid m \\ d>1}} \log \prod_{\ell_p =d}\left(
1+\frac 1{p-1}\right)\nonumber\\
&<&  3.3\sum_{\substack{d\mid m\\
d>1}}\frac{\log\log d}{\phi (d)}
 .\end{eqnarray}

By Lemma \ref{lem3} and $p(m)>x$,
$$|{\cal Q}_m|<\frac {\log ( 2p(m) )}{\log 2}  :=g(m),$$
where ${\cal Q}_m$ is the set of all prime divisors of $m$.

Since $p(m)\ge 173$, similar to the  arguments in the previous
section, we have
\begin{eqnarray*}\sum_{\substack{d\mid m\\
d>1}}\frac{\log\log d}{\phi (d)} &<& \left( 1 +\frac {\log\log
p(m)} {p(m)-1} \frac {p(m)}{p(m)-\log\log p(m)}  \right)^{g(m)}
-1\\
 &:= & T-1.
\end{eqnarray*}
By $p(m)\ge 173$,
\begin{eqnarray*}\log T &=&g(m)\log \left( 1 +\frac {\log\log p(m)} {p(m)-1} \frac {p(m)}{p(m)-\log\log p(m)}
\right)\\
&<& g(m)\frac {\log\log p(m)} {p(m)-1} \frac {p(m)}{p(m)-\log\log p(m)}\\
&=& \frac {\log (2p(m))}{\log 2}  \frac
{\log\log p(m)} {p(m)-1} \frac {p(m)}{p(m)-\log\log p(m)} \\
&=&\frac 1{\log 2}\frac {\log (2p(m))}{\sqrt{p(m)}} \frac{\log\log
p(m)} {\sqrt{p(m)}} \frac{p(m)}{p(m)-1} \frac {p(m)}{p(m)-\log\log
p(m)} \\
&\le & \frac 1{\log 2} \frac {\log 346}{\sqrt{173}} \frac{\log\log
173} {\sqrt{173}} \frac{173}{172} \frac {173}{173-\log\log
173} \\
&<& 0.082.
\end{eqnarray*}
It follows that
$$\sum_{\substack{d\mid m\\
d>1}}\frac{\log\log d}{\phi (d)} < e^{0.082} -1<0.086.$$ By
\eqref{xx1}, we have
\begin{eqnarray*}\log \left( x_1\frac{\phi (zd_1)}{z} \right)&<& 3.3\sum_{\substack{d\mid m\\
d>1}}\frac{\log\log d}{\phi (d)}\\
 &<&3.3 \times 0.086 \\
 &<&
0.3.\end{eqnarray*} Hence
$$x_1\frac{\phi (zd_1)}{z}<e^{0.3}<1.35.$$

 Since $x_1$ and $y_1$ are odd and
$x_1>y_1\ge 1$, it follows that $x_1\ge 3$.

If $d_1\ge 2$, then,  by Lemma \ref{lem4x}, we have $\phi
(zd_1)\ge \phi (2z)$. For $x_1\ge 3$, $1\le z\le 80$ and $d_1\ge
2$, we have
$$x_1\frac{\phi (zd_1)}{z}\ge 3\frac{\phi (2z)}{z}>1.4,$$
 a contradiction. Hence $d_1=1$.
If $x_1\ge 7$, then, by $1\le z\le x-y \le 80$, we have
$$x_1\frac{\phi (zd_1)}{z}\ge 7\frac{\phi (z)}{z}>1.4,$$
 a contradiction. Hence $x_1<7$. Since $x_1$ is odd and $x_1>1$, we have $x_1\in \{ 3, 5\} $.
So $x\le 5$ and $z\le x-y\le 4$. It is easy to see that, for $1\le
z\le 4$,
$$x_1\frac{\phi (zd_1)}{z}\ge 3\frac{\phi (z)}{z}>1.4,$$
 a contradiction.

This completes the proof of Theorem \ref{thm3}.

\section*{Acknowledgments}

This work was supported by the National Natural Science Foundation
of China, Grant No. 11371195 and a project funded by the Priority
Academic Program Development of Jiangsu Higher Education
Institutions.

\end{document}